\pdfoutput=1
\documentclass{article}

\usepackage{bbm}

\usepackage{amsmath, amsfonts, amsthm}
\usepackage{mathrsfs}

\DeclareMathOperator*{\essinf}{ess\,inf}
\DeclareMathOperator*{\mes}{mes}
\DeclareMathOperator*{\defeq}{\mathrel{\mathop:}=}
\newcommand{\abs}[1]{\lvert#1\rvert}
\newcommand{\norm}[1]{\lVert#1\rVert}

\theoremstyle{plain}
\newtheorem{theorem}{\indent Theorem}[section]

\newtheorem{lemma}{\indent Lemma}[section]

\theoremstyle{definition}

\theoremstyle{remark}
\newtheorem{remark}{\indent Remark}[section]

\begin{document}
\title{Life Span of Solutions for a Semilinear
\\
Heat Equation with Initial Data
\\
Non-Rarefied at $\infty$
\footnote{This work is partially supported by NSFC (No. 11371153), Specialized Research Fund for the Doctoral Program of
High Educational Department of China.}}

\author{Zhiyong Wang\footnote{E-mail address: wangzhiyong236@163.com.}, \quad Jingxue Yin\footnote{Corresponding author. E-mail address: yjx@scnu.edu.cn.} \\
School of Mathematical Sciences, South China Normal University, \\
Guangzhou, 510631, PR China
}

\maketitle

\begin{abstract}
  We study the Cauchy problem for a semilinear heat equation with
  initial data non-rarefied at $\infty$. Our interest lies in the
  discussion of the effect of the non-rarefied factors on the life
  span of solutions, and some sharp estimates on the life span is
  established.

  \vspace{2ex}

  \vspace{2ex}
\end{abstract}

\section{Introduction}
Consider the Cauchy problem
\begin{equation}
  \begin{cases}
    u_{t}=\Delta u+\abs{u}^{p-1}u,\quad &
    (x,t) \in \mathbb{R}^{n}\times(0,+\infty),\\
    u(x,0)=\phi(x),&x\in \mathbb{R}^{n},
  \end{cases}
  \label{P}
\end{equation}
where $p>1$, and $\phi$ is a non-negative, bounded and continuous
function in $\mathbb{R}^{n}$, which is not identically equal to zero.

It is well known that the solution $u(x,t)$ of \eqref{P} may blow up
in finite time $T^{\ast}$, that is $\lim\limits_{t\to
{T^\ast}^{-}}\norm{u(\cdot,t)}_{L^\infty}(\mathbb{R}^n)=\infty$, see
\cite{Fu1966, Hu2011}. The finite time $T^{\ast}$ is popularly called to
be the \textit{life span}, which depends heavily on the exponent $p$ and the
properties of the initial datum, and its study has arisen much
attention during recent years, see for example \cite{LN1992, GW1995,
  MY1998, MY2001, OY2011} and references therein.

It is worthy of mentioning that the properties of the initial data in
some neighbourhood of $\infty$ are shown to be crucial factors
affecting the life span of solutions. In this respect, for
$\phi(x)\equiv\lambda\psi(x)$ with $\psi(x)\in C(\mathbb R^n)\cap
L^\infty(\mathbb R^n)$, $\psi(x) \ge 0$, and $\psi(x)$ is positive in
some neighbourhood of $\infty$, the asymptotic behaviour of the life
span $T_\lambda$ seems to be interesting. Indeed, Lee and Ni
\cite{LN1992} proved that if $\liminf\limits_{\abs{x} \to
  \infty}\psi(x)= k > 0$, then $C_1\lambda^{1-p}\le T_\lambda\le
C_2\lambda^{1-p}$ for some positive constants $C_1$ and $C_2$ which
are independent of $\lambda$; while Gui and Wang \cite{GW1995} further
showed that if $\lim\limits_{\abs{x} \to \infty}\psi(x) = k$, then
${\displaystyle\lim_{\lambda\to 0^+}}\lambda^{p-1}T_\lambda=\frac1{p-1}k^{1-p}$.
More recently, Yamauchi \cite{Ya2011} has found that as long as the
initial datum $\phi(x)$ is positive in some {\bf conic neighbourhood}
of $\infty$, the solution does blow up in a finite time $T^*$, and
some elegant estimates on the life span are given.

\vspace{2ex}

The purpose of the present paper is to characterize the role of
rarefaction properties of the initial data at $\infty$, that is the
effect on the life span of solutions of \eqref{P}. We will show that
as long as $\infty$ is not rarefied, although it is permitted for
initial data to have zeros in any conic neighbourhood of $\infty$, the
solutions will blow up definitely.  Of course, we are interested in
estimating on the life span of solutions from the point of view of
density analysis.

Before stating our main results we should recall or introduce some
notations and conceptions. We begin with the definition of rarefaction
point (see \cite[p.~162]{VH1985}): A point $x_{0}$ is a \textit{point
  of rarefaction of the set $E$} if
\begin{equation*}
  \lim_{r \to 0^+} \frac{\mes \bigl( B(x_{0}, r) \cap E \bigr)}{\mes \bigl( B(x_{0},r)  \bigr)} = 0,
\end{equation*}
where $\mes(F)$ is the Lebesgue measure of the set $F$, and $B(x,r)$ is the ball
centered at $x$ with radius $r$. It is reasonable to say that $\infty$ is a
\textit{point of rarefaction of the set $E$} if
\begin{equation*}
  \lim_{r \to +\infty} \frac{\mes \bigl( B(0, r) \cap E \bigr)}{\mes \bigl( B(0,r)  \bigr)} = 0.
\end{equation*}
Alternatively, $\infty$ is called to be a \textit{rarefaction point of
  a non-negative function $\phi$} if for any $\alpha > 0$ it is a
point of rarefaction of the set $\{x \mid \phi(x) \ge \alpha\}$. We
have the following theorem.

\begin{theorem}\label{th:0}
If $\infty$ is not a rarefaction point of the initial datum $\phi$,
then the solution of \eqref{P} blows up in finite time
$T^{\ast}$ with
\begin{equation}
  \label{eq:0}
  T^{\ast} \le \frac{1}{p-1}\inf_{\alpha > 0}\Bigl( \alpha D(\alpha)  \Bigr)^{1-p},
\end{equation}
where
\begin{equation}\label{eq:D}
D(\alpha) \defeq \limsup\limits_{r \to +\infty}\frac{\mes \bigl( \{x \mid
  \phi(x) \ge \alpha \} \cap B(0, r) \bigr)}{\mes \bigl(B(0,r)\bigr)}.
\end{equation}
\end{theorem}

In what follows, we do not intend to give a proof for the above theorem,
but prefer to present a stronger version of Theorem \ref{th:0} in
the following settings. Let $\alpha, r > 0$, denote
\begin{equation*}
  D(\alpha; r) \defeq \sup_{x \in \mathbb{R}^{n}}\frac{\mes \bigl( B(x,r)\cap
    \{y\mid \phi(y)\ge \alpha\} \bigr)}{\mes \bigl( B(x,r)  \bigr)},
\end{equation*}
and define
\begin{equation}
  \overline{D}(\alpha) := \limsup_{r\to +\infty}D(\alpha;r).
\end{equation}

We are now in a position to present the main result of the following theorem.

\begin{theorem}\label{th:main}
  Suppose that there exists $\alpha > 0$ such that $\overline{D}(\alpha) > 0$.
  Then the solution of
  \eqref{P} blows up in finite time $T^{\ast}$ with
  \begin{equation}
    \label{eq:main}
    T^{\ast}\le \frac{1}{p-1}\Bigl( \alpha \overline{D}(\alpha)  \Bigr)^{1-p}.
  \end{equation}
\end{theorem}

We shall give some comments on Theorem \ref{th:main}. First, by the
definition of $\overline{D}(\alpha)$ and $D(\alpha)$, Theorem
\ref{th:main} implies Theorem \ref{th:0} directly. Second, by a simple
comparison argument we obtain the lower bound of the life span:
\begin{equation*}
  \frac{1}{p-1}\norm{\phi}_{L^{\infty}(\mathbb{R}^{n})}^{1-p} \le
  T^{\ast}.
\end{equation*}
Thus Theorem \ref{th:main} can show that the minimal time blow-up
occurs\footnote{If the blow-up time of the solution of the problem
  \eqref{P} is the same as that of the ode: $u'=u^{p}$ with the
  initial value $u(0) = \norm{\phi}_{L^{\infty}(\mathbb{R}^{n})}$, we
  say that the minimal time blow-up occurs, for more details see
  Yamauchi \cite{Ya2011} and references therein.} for some initial data
(for the details see the example in Section \ref{se:contain}).
Finally, we point out that to prove Theorem \ref{th:main} we take good
advantage of basic properties of the heat kernel, so we believe that
the argument is general and can be applied to similar problems posed
on manifolds, e.g., the hyperbolic space, see Remark \ref{re:extand}
and Remark \ref{re:final}.

\vspace{2ex} This paper is organized as follows. In Section
\ref{se:proof} we give the proof of Theorem \ref{th:main}. Subsequently,
we show that Theorem \ref{th:main} implies the  main results of Yamauchi
\cite{Ya2011}, in Section \ref{se:contain}.

\section{Proof of Theorem \ref{th:main}}\label{se:proof}
Before proving Theorem \ref{th:main} we first introduce the basic
properties of the heat kernel in $\mathbb{R}^{n}$, and give a lemma on
the life span of the solutions of \eqref{P}.

\vspace{2ex}

The solution of \eqref{P} can be written as (see \cite[p.51, (17)]{Ev2010}):
\begin{equation*}
  u(x,t)=\int_{\mathbb{R}^{n}}g(x,y,t)\phi(y)\,\mathrm{d}y
  +\int_{0}^{t}\Bigl(\int_{\mathbb{R}^{n}}g(x,y,t-s)\abs{u}^{p-1}u(y,s)\,\mathrm{d}y \Bigr)\,\mathrm{d}s.
\end{equation*}
Here $g$ is the heat kernel in $\mathbb{R}^{n}$, which is a function
of the distance of $x,y\in \mathbb{R}^{n}$ and the time $t$, that is,
\begin{equation}\label{eq:same}
  g(x,y,t)=k\bigl(d(x,y),t\bigr).
\end{equation}

We summarize some elementary properties of the heat kernel in the
following lemma.
\begin{lemma}[{\cite[Section 1.3 and 2.7]{G2009}}]\label{le:basic}
  Let $x, y, z \in \mathbb{R}^{n}$, and $s, t > 0$.
\begin{itemize}
	\item[(i)] $g(x,y,t)=g(y,x,t);$
	\item[(ii)] $g(Tx,Ty,t)=g(x,y,t), \textrm{ where } T \textrm{ is an isometry in }
    \mathbb{R}^{n};$
  \item[(iii)] Semigroup property
    \begin{equation}
      \int_{\mathbb{R}^{n}}g(x,y,t)g(y,z,s)\,\mathrm{d}y =
      g(x,z,s+t);
      \label{eq:semigroup}
    \end{equation}
  \item[(iv)] Conservation of probability
    \begin{equation}
      \int_{\mathbb{R}^{n}}g(x,y,t)\,\mathrm{d}y=1.
      \label{eq:probability}
    \end{equation}
\end{itemize}
\end{lemma}
\begin{remark}\rm
  In $\mathbb{R}^{n}$ it is well known that $g(x,y,t)=\frac{1}{(4\pi
    t)^{n/2}}\exp \bigl(-\frac{\abs{x-y}^{2}}{4t}\bigr)$. In this
  paper we take $T$ to be the translation $T(x) = x + x_{0}$ for fixed
  $x_{0} \in \mathbb{R}^{n}$.
\end{remark}
\begin{remark}\label{re:extand}\rm
  In Lemma \ref{le:basic} the properties (i) -- (iii) are standard for
  the heat kernel and (iv) is satisfied for the manifold that is
  complete and with Ricci curvature bounded below (see \cite[Theorem
  5.2.6]{D1989}). The property \eqref{eq:same}, used only to derive
  the uniform estimate \eqref{eq:2.12}, is satisfied for the
  homogeneous spaces. Thus the proofs of Theorem 2 depend only on the
  basic properties of the heat kernel, which \textit{do not} depend on
  the explicit expression of the heat kernel. Therefore the proofs can
  be extended to the problems posed on other Riemannian manifolds, e.g., on the
  hyperbolic space\footnote{For blow-up problems of semilinear heat equations
  on the hyperbolic space, see the recent works of Bandle, Pozio and
  Tesei~\cite{BPT2011} and Wang and Yin~\cite{WY2014}.}.
\end{remark}

\vspace{2ex}

Now we shall give an a priori estimate on the life span, which was
essentially introduced by Weissler \cite{We1981} for proving the
blow-up of the non-trivial positive solutions of \eqref{P} in the
critical case, see also \cite[Chapter 5]{Hu2011}. For the convenience
of the reader, we prefer to represent it here.
\begin{lemma}\label{le:life}
	Let $p >1$, and $\phi \ge 0$ in $L^{\infty}(\mathbb{R}^{n})$ is not
	identically zero. Suppose that $u$ is a solution of the problem
	\eqref{P} with the initial value $\phi$ in $[0,T^{\ast})$. Then
  \begin{equation}\label{eq:basic-linear}
    \frac{1}{p-1}\Bigl(\sup_{z \in \mathbb{R}^{n}} \int_{\mathbb{R}^{n}}g(z,y,t)\phi(y)\,\mathrm{d}y
    \Bigr)^{1-p}\ge t \quad\text{for}\quad t \in (0,T^{\ast}).
  \end{equation}
In particular, we have the following upper bound estimate on the life span $T^{\ast}$:
\begin{equation*}
  T^{\ast} \le \sup\Bigl\{T > 0 \mid \frac{1}{p-1}\Bigl(\sup_{z \in \mathbb{R}^{n}}
  \int_{\mathbb{R}^{n}}g(z,y,t)\phi(y)\,\mathrm{d}y
  \Bigr)^{1-p}\ge t \quad\text{for}\quad t \in (0,T)\Bigr\}.
\end{equation*}
\end{lemma}
\begin{proof}
  Take $z \in \mathbb{R}^{n}$ and fix $0<t<T^{\ast}$. Let
  $\overline{t} \in [0,t]$. Since $\phi(x) \ge 0$, by the comparison
  principle it follows that $u(x,t) \ge 0$. Thus we have
  \begin{equation}\label{eq:tmild}
    u(x,\overline{t})=
    \int_{\mathbb{R}^{n}}g(x,y,\overline{t})\phi(y)\,\mathrm{d}y
    +\int_{0}^{\overline{t}}\int_{\mathbb{R}^{n}}g(x,y,\overline{t}-s)
    u^{p}(y,s)\,\mathrm{d}y\mathrm{d}s.
  \end{equation}
  Multiplying \eqref{eq:tmild} by $g(x, z, t-\overline{t})$ and
  integrating with respect to $x$ over $\mathbb{R}^{n}$, we obtain
  \begin{multline*}
    \int_{\mathbb{R}^{n}}g(x,z,t-\overline{t})
    u(x,\overline{t})\,\mathrm{d}x=
    \int_{\mathbb{R}^{n}}g(x,z,t-\overline{t})
    \int_{\mathbb{R}^{n}}g(x,y,\overline{t})\phi(y)\,\mathrm{d}y\mathrm{d}x
    \\+\int_{\mathbb{R}^{n}}g(x,z,t-\overline{t})
    \int_{0}^{\overline{t}}\int_{\mathbb{R}^{n}}g(x,y,\overline{t}-s)
    u^{p}(y,s)\,\mathrm{d}y\mathrm{d}s\mathrm{d}x.
  \end{multline*}
  By (i) and (iii) in Lemma \ref{le:basic}, and by Fubini's theorem,
  \begin{multline*}
    \int_{\mathbb{R}^{n}}g(z,x,t-\overline{t})
    u(x,\overline{t})\,\mathrm{d}x= \int_{\mathbb{R}^{n}}g(z,y,t)
    \phi(y)\,\mathrm{d}y\\
    +
    \int_{0}^{\overline{t}}\int_{\mathbb{R}^{n}}g(z,y,t-s)u^{p}(y,s)\,\mathrm{d}y\mathrm{d}s.
  \end{multline*}
  Since $g(z,y,t-s) \ge 0$ and $\int_{\mathbb{R}^{n}}g(z,y,t-s)
  \,\mathrm{d}y = 1$, by Jensen's inequality, the above equation
  implies
  \begin{multline}
    \int_{\mathbb{R}^{n}}g(z,x,t-\overline{t})
    u(x,\overline{t})\,\mathrm{d}x\ge \int_{\mathbb{R}^{n}}g(z,y,t)
    \phi(y)\,\mathrm{d}y\\+
    \int_{0}^{\overline{t}}\Bigl(\int_{\mathbb{R}^{n}}g(z,y,t-s)u(y,s)\,\mathrm{d}y\Bigr)^{p}\mathrm{d}s.
    \label{eq:3-2G}
  \end{multline}
  Denote the right hand side of \eqref{eq:3-2G} by $G(\overline{t})$:
  \begin{equation*}
    G(\overline{t})\defeq
    \int_{\mathbb{R}^{n}}g(z,y,t)
    \phi(y)\,\mathrm{d}y\\+ \int_{0}^{\overline{t}}
    \Bigl(\int_{\mathbb{R}^{n}}g(z,y,t-s)u(y,s)\,\mathrm{d}y\Bigr)^{p}\mathrm{d}s.
  \end{equation*}
  We have $G(t) > 0$ for $t \in [0,T^\ast)$ and
  \begin{equation}
    G(0)=
    \int_{\mathbb{R}^{n}}g(z,y,t)
    \phi(y)\,\mathrm{d}y.
    \label{eq:xx}
  \end{equation}
  Differentiating $G(\overline{t})$ with respect to $\overline{t}$, by \eqref{eq:3-2G} we have the
  inequality
  \begin{equation*}
      G'(\overline{t})=\Bigl(
      \int_{\mathbb{R}^{n}}g(z,y,t-\overline{t})u(y,\overline{t})\,\mathrm{d}y
      \Bigr)^{p}\ge G^{p}(\overline{t});
  \end{equation*}
  that is
  \begin{equation}\label{eq:b4-2.8}
    G^{-p}(\overline{t})G'(\overline{t})\ge 1.
  \end{equation}
  Integrating \eqref{eq:b4-2.8} with respect to $\overline{t}$ over $[0,t]$, we obtain
  \begin{equation*}
    \frac{1}{1-p}\Bigl( G^{1-p}(t)-G^{1-p}(0) \Bigr)\ge t.
  \end{equation*}
  Thus
  \begin{equation}\label{eq:yy}
    \frac{1}{p-1}G^{1-p}(0)\ge t,
  \end{equation}
  since $p>1$. Then the lemma follows from \eqref{eq:xx}
  and \eqref{eq:yy}.
\end{proof}

\vspace{2ex}

Now we are in a position to prove Theorem \ref{th:main}.

\begin{proof}[Proof of Theorem \ref{th:main}] By the assumption of the theorem, for any $\varepsilon > 0$ there
  exist $z_{k}\in \mathbb{R}^{n}$ and $r_{k} > 0$ for $k=1, 2,
  \ldots$, such that
  \begin{equation}\label{eq:2.6}
    \lim_{k \to \infty} r_{k} = +\infty,
  \end{equation}
  and
  \begin{equation}\label{eq:b4-2.10}
    \frac{\mes \bigl( B(z_{k},r_{k})\cap \{y\mid \phi(y)\ge \alpha\}  \bigr)}{\omega_{n}r^{n}_{k}}
    \ge \overline{D}(\alpha) - \varepsilon,
  \end{equation}
  where $\omega_{n}$ is the volume of $n$-dimensional unit ball.

  Take $\bar{r}_{k} = \sqrt{r_{k}}$. We claim the following lemma,
  which will be proved in the end of this section.
  \begin{lemma}\label{le:b4}
    For any $\delta \in (0,1)$, there exists $K \in \mathbb{N}$ such that for
    $k>K$
    \begin{multline}
      \sup_{x \in B(z_{k}, r_{k} - \bar{r}_{k})}\int_{B(z_{k},
        r_{k})}\mathbbm{1}_{B(x, \bar{r}_{k})}(y)g(x, y,
      T^{\ast})\phi(y)\,\mathrm{d}y\\ \ge
      \frac{1-\delta}{(r_{k}-\bar{r}_{k})^{n}}\Bigl[\alpha\bigl(\overline{D}(\alpha)
      - \varepsilon\bigr)r_{k}^{n} -
      \norm{\phi}_{L^{\infty}(\mathbb{R}^{n})}\bigl(r_{k}^{n} -
      (r_{k}-2\bar{r}_{k})^{n}\bigr)\Bigr].
    \end{multline}
    Here $\mathbbm{1}_{E}$ is the characteristic function of the set $E$.
  \end{lemma}

  Lemma \ref{le:b4} implies that for $k > K$
  \begin{multline}\label{eq:b4-2.18}
    \sup_{z \in \mathbb{R}^{n}}
    \int_{\mathbb{R}^{n}}g(z,y,T^{\ast})\phi(y)\,\mathrm{d}y \ge
    \sup_{x \in B(z_{k}, r_{k} - \bar{r}_{k})}\int_{B(z_{k},
      r_{k})}\mathbbm{1}_{B(x, \bar{r}_{k})}(y)
    g(x, y, T^{\ast})\phi(y)\,\mathrm{d}y \\
    \ge
    \frac{1-\delta}{(r_{k}-\bar{r}_{k})^{n}}\Bigl[\alpha\bigl(\overline{D}(\alpha)
    - \varepsilon\bigr)r_{k}^{n} -
    \norm{\phi}_{L^{\infty}(\mathbb{R}^{n})}\bigl(r_{k}^{n} -
    (r_{k}-2\bar{r}_{k})^{n}\bigr)\Bigr].
  \end{multline}
  Let $k \to \infty$. \eqref{eq:2.6} and \eqref{eq:b4-2.18} show
  \begin{equation*}
    \sup_{z \in \mathbb{R}^{n}} \int_{\mathbb{R}^{n}}g(z,y,T^{\ast})\phi(y)\,\mathrm{d}y
    \ge (1-\delta)\alpha\bigl(\overline{D}(\alpha) - \varepsilon\bigr);
  \end{equation*}
  and since $\delta$ and $\varepsilon$ were arbitrary, we obtain
  \begin{equation}\label{eq:2.21}
    \sup_{z \in \mathbb{R}^{n}} \int_{\mathbb{R}^{n}}g(z,y,T^{\ast})\phi(y)\,\mathrm{d}y
    \ge \alpha\overline{D}(\alpha).
  \end{equation}
  Now the theorem follows from \eqref{eq:2.21} and Lemma \ref{le:life}
  immediately.
  \end{proof}

\vspace{2ex}

Finally, we should give the proof of Lemma \ref{le:b4}.
\begin{proof}[Proof of Lemma \ref{le:b4}]
  It is sufficient to show that there exists $K$ satisfies the following property: for $k>K$ there exists
  $x_{k} \in B(z_{k},r_{k}-\bar{r}_{k})$ such that
  \begin{multline}\label{eq:a5}
   \int_{B(z_{k},
        r_{k})}\mathbbm{1}_{B(x_{k}, \bar{r}_{k})}(y)g(x_{k}, y,
      T^{\ast})\phi(y)\,\mathrm{d}y \ge
    \frac{1-\delta}{(r_{k}-\bar{r}_{k})^{n}}\Bigl[\alpha\bigl(\overline{D}(\alpha)
    - \varepsilon\bigr)r_{k}^{n} \\-
    \norm{\phi}_{L^{\infty}(\mathbb{R}^{n})}\bigl(r_{k}^{n} -
    (r_{k}-2\bar{r}_{k})^{n}\bigr)\Bigr].
  \end{multline}
  For any $k=1,2,\cdots$, we define a sequence of functions $F_{k}:
  B(z_{k},r_{k}) \to \overline{\mathbb{R}^{+}}$ as
  \begin{equation*}
    F_{k}(x) \defeq \int_{B(z_{k}, r_{k})}\mathbbm{1}_{B(x, \bar{r}_{k})}(y)g(x, y, T^{\ast})\phi(y)\,\mathrm{d}y.
  \end{equation*}
  Then $F_{k}$ is a non-negative, continuous and bounded function.
  Consider the integral of $F_{k}(x)$ over $B(z_{k},
  r_{k}-\bar{r}_{k})$. By Fubini's theorem, we have
  \begin{equation}\label{eq:2.9}
    \begin{split}
      &\int_{B(z_{k}, r_{k}-\bar{r}_{k})}F_{k}(x)\,\mathrm{d}x =
      \int_{B(z_{k},r_{k})}\mathbbm{1}_{B(z_{k},r_{k}-\bar{r}_{k})}(x)F_{k}(x)\,\mathrm{d}x \\
      =&
      \int_{B(z_{k},r_{k})}\mathbbm{1}_{B(z_{k},r_{k}-\bar{r}_{k})}(x)\Bigl(
      \int_{B(z_{k}, r_{k})}\mathbbm{1}_{B(x, \bar{r}_{k})}(y)g(x, y,
      T^{\ast}
      )\phi(y)\,\mathrm{d}y  \Bigr)\,\mathrm{d}x \\
      = &\int_{B(z_{k},r_{k})} \Bigl( \int_{B(z_{k},r_{k})}
      \mathbbm{1}_{B(z_{k},r_{k}-\bar{r}_{k})}(x) \mathbbm{1}_{B(x,
        \bar{r}_{k})}(y)g(x, y, T^{\ast})\,\mathrm{d}x
      \Bigr)\phi(y)\,\mathrm{d}y\\
      =&\int_{B(z_{k}, r_{k})}I_{k}(y)\phi(y)\,\mathrm{d}y,
    \end{split}
  \end{equation}
  where
  \begin{equation*}
    I_{k}(y) \defeq \int_{B(z_{k},r_{k})}
    \mathbbm{1}_{B(z_{k},r_{k}-\bar{r}_{k})}(x) \mathbbm{1}_{B(x, \bar{r}_{k})}(y)g(x,
    y, T^{\ast})\,\mathrm{d}x.
  \end{equation*}
  Since $g(x,y,t)$ is a function of the distance between $x$ and $y$
  for fixed $t$; and since
  \begin{equation*}
    \int_{\mathbb{R}^{n}}g(x, y, T^{\ast})\,\mathrm{d}x = 1 \quad\text{for}\quad y\in \mathbb{R}^{n},
  \end{equation*}
  it follows that for fixed $\delta > 0$ there exists $R > 0$ such that
  \begin{equation}\label{eq:2.12}
    \int_{\mathbb{R}^{n}}\mathbbm{1}_{B(y,\bar{r}_{k})}(x)g(x, y, T^{\ast})\,\mathrm{d}x
    = \int_{B(y, \bar{r}_{k})}g(x, y, T^{\ast}) \,\mathrm{d}x \ge 1 -\delta
  \end{equation}
  for any $y\in \mathbb{R}^{n}$ and any $\bar{r}_{k} > R$. Since
  $\bar{r}_{k} = \sqrt{r_{k}}$ and $r_{k} \to \infty$ as $k \to
  \infty$, there exists $K$ such that
  \begin{equation*}
    \bar{r}_{k} > R \quad\text{for}\quad k > K.
  \end{equation*}
  Thus by \eqref{eq:2.12} we obtain
  \begin{equation}
    \label{eq:2.16b5}
    \int_{\mathbb{R}^{n}}\mathbbm{1}_{B(y,\bar{r}_{k})}(x)g(x, y, T^{\ast})\,\mathrm{d}x
    \ge 1 - \delta \quad\text{for}\quad k > K.
  \end{equation}
  In the following proof we shall always take $k > K$. Moreover, for
  $y \in B(z_{k}, r_{k}-2 \bar{r}_{k})$, it is easily seen that
  \begin{equation}\label{eq:2.13}
    \mathbbm{1}_{B(y,\bar{r}_{k})}(x) = \mathbbm{1}_{B(x, \bar{r}_{k})}(y)
    \quad\text{and}\quad B(y, \bar{r}_{k}) \subset B(z_{k}, r_{k}-\bar{r}_{k}).
  \end{equation}
  Then by \eqref{eq:2.13} and \eqref{eq:2.12} we obtain, for $y \in
  B(z_{k}, r_{k}-2 \bar{r}_{k})$,
  \begin{equation*}
    \begin{split}
      I_{k}(y) &= \int_{B(z_{k}, r_{k}-\bar{r}_{k})}\mathbbm{1}_{B(x,\bar{r}_{k})}(y)g(x, y, T^{\ast})\,\mathrm{d}x\\
      &= \int_{B(z_{k}, r_{k}-\bar{r}_{k})}\mathbbm{1}_{B(y,\bar{r}_{k})}(x)g(x,y,T^{\ast})\,\mathrm{d}x\\
      &= \int_{\mathbb{R}^{n}}\mathbbm{1}_{B(y,\bar{r}_{k})}(x)g(x, y,
      T^{\ast})\,\mathrm{d}x \ge 1-\delta.
    \end{split}
  \end{equation*}
  Thus by the above inequality and the definition of $I_{k}$ we have
  \begin{equation}\label{eq:b1a}
    I_{k}(y)\ge
    \begin{cases}
      1-\delta, \quad &y \in B(z_{k}, r_{k}-2 \bar{r}_{k}),\\
      0,& y \in B(z_{k}, r_{k}) \setminus B(z_{k},
      r_{k}-2 \bar{r}_{k}).
    \end{cases}
  \end{equation}
  By \eqref{eq:2.9} and \eqref{eq:b1a} we obtain
  \begin{equation*}
    \begin{split}
      &\int_{B(z_{k}, r_{k}-\bar{r}_{k})}F_{k}(x)\,\mathrm{d}x
      = \int_{B(z_{k}, r_{k})}I_{k}(y)\phi(y)\,\mathrm{d}y \\
      & =\int_{B(z_{k},
        r_{k}-2\bar{r}_{k})}I_{k}(y)\phi(y)\,\mathrm{d}y
      + \int_{B(z_{k}, r_{k})\setminus B(z_{k}, r_{k}-2\bar{r}_{k})}I_{k}(y)\phi(y)\,\mathrm{d}y\\
      &\ge \int_{B(z_{k},
        r_{k}-2\bar{r}_{k})}I_{k}(y)\phi(y)\,\mathrm{d}y \ge
      (1-\delta)\int_{B(z_{k}, r_{k}-2
        \bar{r}_{k})}\phi(y)\,\mathrm{d}y,
    \end{split}
  \end{equation*}
  which, together with the continuity of $F_{k}(x)$, implies that there
  exists $x_{k} \in B(z_{k}, r_{k}-\bar{r}_{k})$ such that
  \begin{equation}\label{eq:r1-2.21}
    F_{k}(x_{k}) \ge \frac{(1-\delta)\int_{B(z_{k}, r_{k}-2 \bar{r}_{k})}\phi(y)\,\mathrm{d}y}
    {\omega_{n}(r_{k} - \bar{r}_{k})^{n}}.
  \end{equation}
  Since \eqref{eq:b4-2.10} implies that
  \begin{multline*}
    \int_{B(z_{k}, r_{k})}\phi(y) \,\mathrm{d}y \ge \alpha \int_{
      B(z_{k},r_{k})\cap \{y\mid \phi(y)\ge \alpha\} } \,\mathrm{d}y\\
    \ge \alpha\mes \bigl( B(z_{k},r_{k})\cap \{x\mid \phi(x)\ge
    \alpha\} \bigr) \ge \alpha\bigl(\overline{D}(\alpha) -
    \varepsilon\bigr)\omega_{n}r_{k}^{n},
  \end{multline*}
  we obtain
  \begin{multline}\label{eq:2.17}
    \int_{B(z_{k}, r_{k}-2 \bar{r}_{k})}\phi(y)\,\mathrm{d}y =
    \int_{B(z_{k}, r_{k})}\phi(y)\,\mathrm{d}y -
    \int_{B(z_{k}, r_{k}) \setminus B(z_{k}, r_{k}-2 \bar{r}_{k})}\phi(y)\,\mathrm{d}y \\
    \ge \alpha\bigl(\overline{D}(\alpha) -
    \varepsilon\bigr)\omega_{n}r_{k}^{n} -
    \norm{\phi}_{L^{\infty}(\mathbb{R}^{n})}\omega_{n}\bigl(r_{k}^{n}
    - (r_{k}-2\bar{r}_{k})^{n}\bigr).
  \end{multline}
  Then by \eqref{eq:r1-2.21} and \eqref{eq:2.17}
  \begin{multline*}\label{eq:2.18}
    F_{k}(x_{k}) \ge
    \frac{1-\delta}{(r_{k}-\bar{r}_{k})^{n}}\Bigl[\alpha\bigl(\overline{D}(\alpha)
    - \varepsilon\bigr)r_{k}^{n} -
    \norm{\phi}_{L^{\infty}(\mathbb{R}^{n})}\bigl(r_{k}^{n} -
    (r_{k}-2\bar{r}_{k})^{n}\bigr)\Bigr],
  \end{multline*}
  which is exactly \eqref{eq:a5} by the definition of $F_{k}$. This
  completes the proof.
  \end{proof}

\begin{remark}\label{re:final}\rm
  Since in the hyperbolic space $\mathbb{H}^{n}$ the volume of the
  ball of radius $r$ is
  $\sigma^{n-1}(\mathbb{S}^{n-1})\int_{0}^{r}\sinh^{n-1}\eta
  \,\mathrm{d}\eta$, where $\sigma^{n-1}(\mathbb{S}^{n-1})$ is the
  surface area of the $n-1$ Euclidean sphere of radius $1$, see
  \cite[p. 79]{R2006}; and since the heat kernel in the hyperbolic
  space satisfies all the corresponding properties in Lemma
  \ref{le:basic} and \eqref{eq:same}, we can modify the above proofs
  to adapt the similar problem posted in the hyperbolic space, say,
  replacing the $\Delta$ in \eqref{P} with the Laplace-Beltrami
  operator in $\mathbb{H}^{n}$.
\end{remark}

\section{Proving Theorem \ref{th:Ya2011} by  Theorem \ref{th:main}}\label{se:contain}

In this section we shall show that Theorem \ref{th:main} implies the
previous results of Yamauchi in \cite{Ya2011}. To state Yamauchi's
results precisely we recall some notations in \cite{Ya2011}. For $\xi'
\in \mathbb{S}^{n-1}$ and $\delta \in (0, \sqrt{2})$, we set conic
neighbourhood $\Gamma_{\xi'}(\delta)$:
\begin{equation*}
  \Gamma_{\xi'}(\delta) = \{\eta \in \mathbb{R}^{n}\setminus\{0\}
  \mid \bigl\lvert{\xi' - \frac{\eta}{\abs{\eta}}}\bigr\rvert < \delta \},
\end{equation*}
and set
\begin{equation}
  \label{eq:r2-3.0}
  S_{\xi'}(\delta)=\Gamma_{\xi'}(\delta)\cap
  \mathbb{S}^{n-1}.
\end{equation}
Define
\begin{equation*}
  \phi_{\infty}(x') = \liminf_{r\to +\infty}\phi(rx') \quad\text{for}\quad x' \in \mathbb{S}^{n-1}.
\end{equation*}
Yamauchi proved the following theorem.
\begin{theorem}[{\cite[Theorem~1 and
    Theorem~2]{Ya2011}}]\label{th:Ya2011}
  Let $n\ge 2$. Assume that there exist $\xi' \in \mathbb{S}^{n-1}$
  and $\delta > 0$ such that $\essinf_{x' \in
    S_{\xi'}(\delta)}\phi_{\infty}(x') > 0$. Then the classical
  solution for \eqref{P} blows up in finite time, and the blow-up time
  is estimated as
  \begin{equation}
    T^{\ast}\le \frac{1}{p-1}\Bigl(\essinf_{x' \in S_{\xi'}(\delta)}\phi_{\infty}(x') \Bigr)^{1-p}.
  \end{equation}
  Let $n = 1$. Assume that $\max\{\liminf\limits_{x\to +\infty}
  \phi(x), \liminf\limits_{x \to -\infty}\phi(x)\} > 0$. Then the
  classical solution for \eqref{P} blows up in finite time, and the
  blow-up time is estimated as
  \begin{equation}
    T^{\ast} \le \frac{1}{p-1}\Bigl(\max\{\liminf\limits_{x\to +\infty}
    \phi(x), \liminf\limits_{x \to -\infty}\phi(x) \}\Bigr)^{1-p}.
  \end{equation}
\end{theorem}

It is easily seen that Theorem \ref{th:Ya2011} implies the upper
bounded estimate of Gui and Wang in \cite{GW1995} immediately; that
is, if $\lim\limits_{\abs{x} \to \infty}\psi(x) = k$, then
$\lambda^{p-1} T_{\lambda} \le \frac{1}{p-1}k^{1-p}$. However, Theorem
\ref{th:Ya2011} can not be applied to some simple cases, which can be illustrated
by the following example.

Take a sequence $\{a_{k}\}_{k=1}^{\infty}$ by $a_{k} = k!$.
We see that $\lim\limits_{k\to\infty}\dfrac{a_{k}}{a_{k+1}} = 0$.
Using the sequence $\{a_{k}\}_{k=1}^{\infty}$ we can construct a
function $\Phi\in C(\mathbb{R}^{n})$, which satisfies $0 \le \Phi(x)
\le 1$ and
\begin{equation*}
  \Phi(x) =
  \begin{cases}
  0, \quad \abs{x} \in [a_{2k-1}+ \frac{1}{4}, a_{2k}-\frac{1}{4}], \\
  1, \quad \abs{x} \in [a_{2k}, a_{2k+1}].
  \end{cases}
\end{equation*}
By the definition of $D(\alpha)$ in \eqref{eq:D} it follows that
\begin{multline*}
  D(1) \ge \limsup_{k\to +\infty} \frac{\mes \bigl( \{x \mid \phi(x)
    \ge 1 \} \cap B(0, a_{2k+1}) \bigr)}{\mes
    \bigl(B(0,a_{2k+1})\bigr)} \ge \limsup_{k\to +\infty}
  \frac{a^{n}_{2k+1} - a^{n}_{2k}}{a^{n}_{2k+1}} =1.
\end{multline*}
Thus for the initial datum $u_{0}(x) = \Phi(x)$, by Theorem \ref{th:0}
the life span $T^{\ast}$ of the solution $u$ can be estimated by
\begin{equation*}
  T^{\ast} \le \frac{1}{p-1}\bigl( 1D(1)\bigr)^{1-p} = \frac{1}{p-1}.
\end{equation*}
Since $v(x,t) = \bigl( 1 - (p -1)t \bigr)^{\frac{1}{p-1}}$, which
blows up at $T = \frac{1}{p-1}$, is an upper solution of \eqref{P}
with $v(x, 0) \ge u(x,0)$, by the comparison principle it follows that
$T^{\ast} \ge T = \frac{1}{p-1}$. Thus $T^{\ast} = \frac{1}{p-1}$,
which shows that the minimal time blow-up occurs for the initial datum
$\Phi(x)$.

\vspace{2ex}
Now we use Theorem \ref{th:main} to prove Theorem \ref{th:Ya2011}.

\begin{proof}[Proof of Theorem \ref{th:Ya2011}]
  We shall prove the case $n =
1$ and $n \ge 2$, respectively.

\vspace{2ex} (i) $n =1$. Let us assume that
\begin{equation*}
  A=\liminf_{x\to +\infty} \phi(x) \ge \liminf_{x \to -\infty}\phi(x).
\end{equation*}
Then for any $\varepsilon > 0$ there exists $R > 0$ such that
\begin{equation*}
  \phi(x) \ge A -\varepsilon \quad\text{for}\quad x > R.
\end{equation*}
Hence $\overline{D}(A - \varepsilon) =1$. By Theorem \ref{th:main} we
obtain
\begin{equation*}
  T^{\ast} \le \frac{1}{p-1}(A - \varepsilon)^{1-p}.
\end{equation*}
Since $\varepsilon > 0$ was arbitrary, it follows
\begin{equation*}
  T^{\ast} \le \frac{1}{p-1}A^{1-p}.
\end{equation*}
The proof for $n =1$ is finished.

\vspace{2ex} (ii) $n \ge 2$. Let $A = \essinf_{x' \in
  S_{\xi'}(\delta)}\phi_{\infty}(x')$. It is sufficient to show that,
for any $\overline{R}>0$, $0<\tau<1$ and $0 < \varepsilon < A$, there
exists a ball $B(x,r)$ with $r>\overline{R}$ such that
\begin{equation}\label{eq:a.4}
  \frac{\mes \Bigl( \{y\mid \phi(y)\ge A-\varepsilon\} \cap B(x,r) \Bigr)}{\omega_{n}r^{n}} > 1-\tau.
\end{equation}
We shall show \eqref{eq:a.4} in three steps.

\vspace{2ex}
First, we show that there exists $R_{1}> 0$ such that
\begin{equation}\label{eq:dens}
  \frac{\sigma^{n-1}\Bigl( \{x' \in \mathbb{S}^{n-1} \mid \phi(rx') \ge A-\varepsilon
    \quad\text{for}\quad r > R_{1}\} \cap S_{\xi'}(\delta) \Bigr)}{\sigma^{n-1}(S_{\xi'}(\delta))} > 1 -  \tau,
\end{equation}
where $\sigma^{n-1}(M)$ is the spherical measure for the measurable
set $M \subset \mathbb{S}^{n-1}$ and $S_{\xi'}(\delta)$ is defined by \eqref{eq:r2-3.0}.

We denote
\begin{equation*}
   S(R)\defeq \{x' \in \mathbb{S}^{n-1} \mid \phi(rx') \ge A-\varepsilon
  \quad\text{for}\quad r > R\},
\end{equation*}
and claim that the set
\begin{equation*}
S(R)\cap S_{\xi'}(\delta)
\end{equation*}
is measurable on $\mathbb{S}^{n-1}$. Indeed, since the function
$\phi(x)$ is continuous, we may view $\phi$ as a continuous function
on $\mathbb{S}^{n-1} \times \overline{R^{+}}$, which implies that, for
fixed $r > 0$, $\phi(rx')$ is continuous on $\mathbb{S}^{n-1}$. Then,
for any $r > 0$, the set
\begin{equation*}
  \{x' \in \mathbb{S}^{n-1} \mid \phi(rx') \ge A-\varepsilon\}
\end{equation*}
is closed on $\mathbb{S}^{n-1}$. Since
\begin{equation*}
  \begin{split}
    S(R)&=\{x' \in \mathbb{S}^{n-1} \mid \phi(rx') \ge A-\varepsilon
    \quad\text{for}\quad r > R\} \\&= \bigcap_{r>R}\{x' \in
    \mathbb{S}^{n-1} \mid \phi(rx') \ge A-\varepsilon\},
\end{split}
\end{equation*}
it follows that the set $S(R)$ is closed. Hence $S(R)\cap
S_{\xi'}(\delta)$ is measurable on $\mathbb{S}^{n-1}$, and the claim
is proved.

By the assumption of
Theorem \ref{th:Ya2011}, we see that
\begin{equation*}
  S_{\xi'}(\delta)  = \bigcup_{k=1,2,\ldots}\bigl(  S(k)\cap  S_{\xi'}(\delta)  \bigr).
\end{equation*}
Noticing that
\begin{equation*}
  S(k_{1})\cap S_{\xi'}(\delta) \subset S(k_{2})\cap S_{\xi'}(\delta)
  \quad\text{for}\quad k_{1} < k_{2},
\end{equation*}
by the standard convergence theorem for a sequence of measurable sets
(see, e.g., \cite[Theorem 11.3]{R1976}), we obtain
\begin{equation*}
  \sigma^{n-1}\bigl( S_{\xi'}(\delta)\bigr) = \lim_{k\to \infty} \sigma^{n-1}\bigl(  S(k)\cap  S_{\xi'}(\delta)  \bigr).
\end{equation*}
Thus, for any $0 < \tau < 1$, there exists $K\in \mathbb{N}$ such that
\begin{equation*}
  \frac{ \sigma^{n-1}\bigl( S(k)\cap S_{\xi'}(\delta) \bigr)}{
  \sigma^{n-1}\bigl( S_{\xi'}(\delta) \bigr)} > 1-\tau
\quad\text{for}\quad k\ge K.
\end{equation*}
We can take $R_{1} = K$ to ensure \eqref{eq:dens}.

\vspace{2ex}
Next, we show that there exists a ball $B(x_{0}, r_{0})$ such that
\begin{equation}\label{eq:b4-3.3}
  B(x_{0},r_{0}) \subset \{x\in \mathbb{R}^{n} \mid R_{1}<\abs{x}<R_{1}+1
  \quad\text{and}\quad \frac{x}{\abs{x}} \in  S_{\xi'}(\delta)\},
\end{equation}
and
\begin{equation}\label{eq:smallball}
  \frac{\mes \bigl(\{x\in B(x_{0},r_{0}) \mid \frac{x}{\abs{x}} \in S(R_{1}) \}\bigr)}
  {\mes \bigl(B(x_{0},r_{0})\bigr)} \ge 1-\tau.
\end{equation}
For this we need the following covering lemma, where
$\overline{B(x,r)}$ is the closure of the ball $B(x,r)$.

\begin{lemma}[{\cite[Theorem 2.7]{M2000}} Besicovitch Covering Theorem]\label{le:besicovitch}
  Suppose $\mu$ is a Borel measure on $\mathbb{R}^{n}$ , $E \subset
  \mathbb{R}^{n}$ , $\mu(E) < \infty$, $\mathscr{F}$ is a collection
  of non-trivial closed balls, and $\inf\{r\mid \overline{B(x,r)}\in
  \mathscr{F}\} = 0$ for all $x\in E$. Then there is a countable
  disjoint sub-collection of $\mathscr{F}$ that covers $\mu$ almost
  all of $E$.
\end{lemma}

We choose $\mu$ to be the Lebesgue measure on $\mathbb{R}^{n}$,
\begin{equation*}
  E= \{x\in \Omega \mid \frac{x}{\abs{x}} \in
  S(R_{1})\},
\end{equation*}
and take
\begin{equation*}
  \mathscr{F}=\{\overline{B(x,r)} \subset \Omega \mid x\in E, r>0\},
\end{equation*}
where $\Omega = \{x\in \mathbb{R}^{n} \mid R_{1}<\abs{x}<R_{1}+1,
\frac{x}{\abs{x}} \in S_{\xi'}(\delta)\ \}$ is an open set of
$\mathbb{R}^{n}$ by the definition of $S_{\xi'}(\delta)$.  Noticing
that the map $\mathbb{R}^{n}\setminus \{0\} \to \mathbb{S}^{n-1}$,
defined by $x \to \frac{x}{\abs{x}}$, is continuous and $S(R_{1})$ is
a closed set of $\mathbb{S}^{n-1}$, we see that $E$ is a measurable
set of $\mathbb{R}^{n}$. Then by Lemma \ref{le:besicovitch} we can
take a sequence of balls $\{\overline{B(x_{i},r_{i})}\}_{i=1}^{\infty} \subset
\mathscr{F}$ such that
\begin{equation}\label{eq:a.10}
  \overline{B(x_{i},r_{i})}\cap \overline{B(x_{j},r_{j})} = \emptyset \quad\text{for}\quad i\neq j,
\end{equation}
and
\begin{equation}\label{eq:b.8}
  \mes\bigl(E\setminus \bigcup_{i} \overline{B(x_{i},r_{i})}\bigr) = 0.
\end{equation}
Thus by \eqref{eq:dens} and the definition of $E$ we obtain
\begin{equation}\label{eq:cover}
  \begin{split}
    \mes(E)& = \int_{R_{1}}^{R_{1}+1}r^{n-1}
    \int_{S_{\xi'}(\delta)}\mathbbm{1}_{S(R_{1}) \cap S_{\xi'}(\delta )}(\xi)\,\mathrm{d}\xi\mathrm{d}r\\
    & \ge (1-\tau) \int_{R_{1}}^{R_{1}+1}r^{n-1}\int_{S_{\xi'}(\delta)}\,\mathrm{d}\xi\mathrm{d}r\\
    & = (1-\tau)\mes(\Omega).
  \end{split}
\end{equation}
By \eqref{eq:a.10}, \eqref{eq:b.8} and the definition
of $\mathscr{F}$, we have
\begin{equation}\label{eq:b4-3.8}
\mes(E) = \mes \Bigl( \bigcup_{i} \bigl( \overline{B(x_{i},r_{i})} \cap E \bigr) \Bigr)
= \sum_{i} \mes \Bigl( \overline{B(x_{i},r_{i})} \cap E \Bigr)
\end{equation}
and
\begin{equation}\label{eq:b4-3.9}
\mes(\Omega) \ge \mes\Bigl(\bigcup_{i}\overline{B(x_{i},r_{i})}\Bigr) =
\sum_{i}\mes\Bigl(\overline{B(x_{i},r_{i})}\Bigr).
\end{equation}
Inserting \eqref{eq:b4-3.8} and \eqref{eq:b4-3.9} into
\eqref{eq:cover}, we obtain
\begin{equation*}
  \sum_{i} \mes \Bigl(\overline{B(x_{i},r_{i})} \cap E \Bigr) \ge (1-\tau)\sum_{i}\mes
  \Bigl(\overline{B(x_{i},r_{i})}\Bigr).
\end{equation*}
Then we can get a ball $B(x_{0},r_{0})$ in the sequence
$\{{B(x_{i},r_{i})}\}_{i=1}^{\infty}$ satisfying \eqref{eq:b4-3.3} and
\eqref{eq:smallball}.

\vspace{2ex}
Finally, we complete the proof by scaling the ball $B(x_{0}, r_{0})$.
We claim that for $x \in \mathbb{R}^{n}$, $r > 0$ and $\lambda > 1$, it follows that
\begin{equation}\label{eq:b5-1}
  \frac{\mes \Bigl(\{y \in \mathbb{R}^{n} \mid \phi(y) \ge A-\varepsilon\} \cap B(\lambda x, \lambda r)   \Bigr)}
  {\mes \Bigl(B(\lambda x,\lambda r)\Bigr)} \ge \frac{\mes \Bigl( B(x,r) \cap E  \Bigr)}{\mes \Bigl(B( x, r)\Bigr)}.
\end{equation}
The claim can be seen by the following argument. By the definitions of
$E$ and $S(R_{1})$ it follows that if $z \in B(x, r)\cap E$ then
$\phi(\lambda z) \ge A -\varepsilon$ and $\lambda z \in B(\lambda x,
\lambda r)$ for $\lambda > 1$. Thus
\begin{equation*}
  \{y \in \mathbb{R}^{n} \mid
  \phi(y) \ge A-\varepsilon\} \cap B(\lambda x, \lambda r) \supset \{\lambda z \mid z \in B(x,r) \cap E \}.
\end{equation*}
Hence
\begin{multline*}
  \mes \Bigl(\{y \in \mathbb{R}^{n} \mid \phi(y) \ge A-\varepsilon\}
  \cap B(\lambda x, \lambda r) \Bigr)\\ \ge \mes \Bigl( \{\lambda z
  \mid z \in B(x,r) \cap E \Bigr) = \lambda^{n}\mes \Bigl( B(x,r) \cap
  E \Bigr),
\end{multline*}
and the claim follows. By \eqref{eq:b4-3.3} we see that $r_{0} < 1$, thus
 $\frac{\overline{R} + 1}{r_{0}} > 1$.
Scaling the ball $B(x_{0}, r_{0})$ with
$\lambda = \frac{\overline{R} + 1}{r_{0}}$, by \eqref{eq:b5-1} and
\eqref{eq:smallball} we obtain
\begin{equation*}
  \frac{\mes \Bigl(\{y \in \mathbb{R}^{n} \mid \phi(y) \ge
    A-\varepsilon\} \cap
    B(\frac{\overline{R}+1}{r_{0}}x_{0},\overline{R}+1) \Bigr)} {\mes
    \Bigl(B(\frac{\overline{R}+1}{r_{0}}x_{0},\overline{R}+1) \Bigr)}
   \ge \frac{\mes \Bigl( B(x_{0},r_{0}) \cap E \Bigr)}{\mes \Bigl(B(
    x_{0}, r_{0})\Bigr)} \ge 1-\tau.
\end{equation*}
Thus the ball $B(\frac{\overline{R}+1}{r_{0}}x_{0},\overline{R}+1)$
satisfies \eqref{eq:a.4}, and the proof is complete.
\end{proof}

\def\cprime{$'$}
\providecommand{\bysame}{\leavevmode\hbox to3em{\hrulefill}\thinspace}

\end{document}